\documentclass[12pt,letterpaper,reqno]{amsart}

\usepackage{times}
\usepackage[T1]{fontenc}
\usepackage{mathrsfs}
\usepackage{latexsym}
\usepackage[dvips]{graphics}
\usepackage{epsfig}
\usepackage{amsmath,amsfonts,amsthm,amssymb,amscd}
\input amssym.def
\input amssym.tex

\newcommand\be{\begin{equation}}
\newcommand\ee{\end{equation}}
\newcommand\bea{\begin{eqnarray}}
\newcommand\eea{\end{eqnarray}}
\newcommand\bi{\begin{itemize}}
\newcommand\ei{\end{itemize}}
\newcommand\ben{\begin{enumerate}}
\newcommand\een{\end{enumerate}}
\newcommand\bc{\begin{center}}
\newcommand\ec{\end{center}}
\newcommand\ba{\begin{array}}
\newcommand\ea{\end{array}}

\newtheorem{thm}{Theorem}[section]

\newtheorem{lem}[thm]{Lemma}

\theoremstyle{definition}
\newtheorem{rek}[thm]{Remark}

\begin{document}

\title{Essentialities in additive bases}

\author{Peter Hegarty}
\email{hegarty@math.chalmers.se} \address{Mathematical Sciences,
Chalmers University Of Technology and G\"oteborg University,
G\"oteborg, Sweden}


\date{\today}

\begin{abstract} Let $A$ be an asymptotic basis for 
$\mathbb{N}_0$ of some order. By an {\em essentiality} of $A$
one means a subset $P$ such that $A \backslash P$
is no longer an asymptotic 
basis of any order and such that $P$ is minimal among all
subsets of $A$ with this property. A finite essentiality of $A$ is called an
{\em essential subset}. In a recent paper, Deschamps and 
Farhi asked the following two questions : (i) does every asymptotic basis of 
$\mathbb{N}_0$ possess some essentiality ? (ii) is the number of 
essential subsets of size at most $k$ of an asymptotic 
basis of order $h$ bounded
by a function of $k$ and $h$ only (they showed the number is always finite) ?
We answer the latter question in the affirmative, and the former in the
negative by means of an explicit construction, for every integer $h \geq 2$, of
an asymptotic basis of order $h$ with no essentialities. 
\end{abstract}


\maketitle

\setcounter{equation}{0}

\setcounter{equation}{0}

\section{Introduction}

Let $A \subseteq \mathbb{N}_0$ such that $0 \in A$, 
and $h \geq 2$ an integer. The $h$-fold
sumset of $A$, denoted $hA$, is the subset of $\mathbb{N}_0$ 
consisting of all possible sums of $h$-tuples of
elements of $A$, i.e.:
\be
hA = \{a_1 + \cdots + a_h : a_1,...,a_h \in A\}.     
\ee
We say that $A$ is a {\em basis} (resp. {\em asymptotic basis}) 
for $\mathbb{N}_0$ of order $h$ if the 
difference set $\mathbb{N}_0 \backslash hA$ is empty (resp. 
finite), in other
words, if every (resp. every sufficiently large) non-negative 
integer can be expressed as a sum of at 
most $h$ non-zero elements  of $A$. This is a 
fundamental notion in additive number theory. In the rest of this article,
we will be concerned only with asymptotic bases and we will refer to these
simply as $\lq$bases'. We hope no confusion arises, for those
who are acquainted with the classical terminology. It is important to note
that much of what we discuss has also been the subject of investigation for
ordinary bases, in which case many things actually beome simpler. 
\\
\\
\noindent
In \cite{4} Nathanson 
introduced the following idea : a subset $P$ of a 
basis $A$ of order 
$h$ is said to be {\em necessary} if $A\backslash P$ is no longer a 
basis of order $h$. Nathanson was concerned with so-called 
{\em minimal bases}, 
that is, bases in which every element is necessary. There is by now an 
extensive literature on these : 
see \cite{3} for a state-of-the-art
result. The paper \cite{1} provides a recent 
perepective on another
popular line of research on the subject of necessary subsets of bases. A very
similar notion was the subject of another recent paper of Deschamps and 
Farhi \cite{2}.
They call a subset $P$ of a basis $A$ (of some order) an 
{\em essentiality} if $A\backslash P$ is no longer a basis, of 
any order, and $P$ is minimal, with respect to inclusion, among subsets of
$A$ with this property. A finite essentiality is called an {\em essential
subset}, and, in the case of a singleton set, an {\em essential element}. 
The main difference between the notions of $\lq$necessary' and $\lq$essential' 
subset of a basis is that removing one of the former need only
increase the order of the basis, whereas removing one of the letter destroys 
the basis property entirely. The reader is invited to stop for a moment
here though and, to avoid later confusion, note the additional subtle 
differences between the meanings of the terms as employed by their various
inventors.
\\
\\
\noindent
The main
result of \cite{2}, improving on the work of earlier authors, provides
a tight upper bound on the number of essential elements in a basis, 
purely in terms of the order of the latter. In particular, this number is 
always
finite. They also show that every
basis possesses only finitely many essential subsets. To achieve
the latter result, they first prove that for any basis $A$ there 
exists a largest positive integer $a=a(A)$ such that $A$ is contained, from 
some point onwards, in an arithmetic progression of common difference $a$. 
They then bound the number of essential subsets in terms of the so-called 
{\em radical} of $a$, that is, the number of distinct primes dividing it.
In contrast to the situation with essential elements, 
this does not yield a bound purely in terms of the order of the basis, and
they give an example to show that no such bound is possible.
\\
\\
\noindent
At the end of their paper, Deschamps and Farhi pose two problems. The first is
whether one can give a universal bound on the number of essential 
subsets of size at most $k$, in a basis of order $h$, which is a function 
of $k$ and $h$ only. Their second problem concerns infinite essentialites. 
First note that it is easy to give examples of bases possessing no essential
subsets whatsoever. As an extreme case, 
take $A = \mathbb{N}_0$ itself, which is a basis of order 1. Nevertheless, 
here we can 
still clearly identify infinite essentialities of $A$, namely the 
complements of the sets $p\mathbb{N}$, where $p$ is any prime. The 
second problem posed in \cite{2} 
is whether every basis for $\mathbb{N}_0$ must contain some essentiality, 
albeit possibly infinite. In fact, they went further and asked more (it
is not important here to recall exactly what),
suggesting possibly that they believed the answer to the basic question was 
yes. 
\\
\\
\noindent
In this paper we solve both problems. In Section 2 we will prove a 
bound of the desired form on the number of essential subsets of bounded
size in a basis of a given order. The proof uses ideas developed in 
\cite{2} and is very short. In Section 3, we shall explicitly construct, 
for every $h \geq 2$, a basis for $\mathbb{N}_0$ of order $h$ without 
essentialities. 

\setcounter{equation}{0}
\section{The number of essential subsets in a basis}

The following facts are proven in \cite{2} : 

\begin{lem}
(i) Let $A$ be a basis (of some order) and $P$ an essentiality of 
$A$. Let $d = d(P)$ be the largest integer such that $A\backslash P$ is
contained in an arithmetic progression of common difference $d$. Then 
$d \geq 2$.
\\
(ii) Let $P_1$ and $P_2$ be any two essentialities of $A$ such that
$P_1 \cup P_2 \neq A$. Then $d(P_1)$ and $d(P_2)$ are relatively prime.
\end{lem}
\noindent
We denote by $p_n \#$ the product of the first $n$ prime numbers :
this is fairly standard notation. We can now prove

\begin{thm}
Let $k,h$ be two positive integers. There exists an integer $\phi(k,h)$
such that any basis for $\mathbb{N}_0$ of order $h$ contains at most 
$\phi(k,h)$ essential subsets of size at most $k$.
\end{thm}

\begin{proof}
Let a basis $A$ of order $h$ be given. Let $P_1,...,P_\phi$ be the complete
list of its essential subsets of size at most $k$ 
(we know from \cite{2} that this list is finite).
Let $d_i := d(P_i)$ be the integers defined in Lemma 2.1, for 
$i = 1,...,\phi$. Since each $P_i$ is a finite set we conclude from the 
lemma that the integers $d_1,...,d_{\phi}$ are pairwise relatively prime. 
Let $X := \cup_{i=1}^{\phi} P_i$ and $t := \prod_{i=1}^{\phi} d_i$. Then 
\be
|X| \leq k\phi,
\ee
$A\backslash X$ is contained in an arithmetic progression of common difference
$t$ and 
\be
t \geq p_{\phi} \#. 
\ee
Let $Y := X \cup \{0\}$. Now since $A$ is a basis of order $h$, the 
$h$-fold sumset $hY$ must meet all congruence classes mod $t$. 
At the very least this implies that
\be
|Y| \geq t^{1/h}.
\ee
Eqs. (2.1)-(2.3) imply that 
\be
k \phi  + 1 \geq (p_{\phi} \#)^{1/h},
\ee
which clearly yields an upper bound on $\phi$ depending
only on $k$ and $h$.
\end{proof}

\begin{rek}
It follows from the prime number theorem that 
\be
p_n \# = \exp ((1+o(1)) \cdot n \log n).
\ee
From this and (2.4) we can easily deduce explicit upper bounds.
For example, if $h$ is fixed and $k \rightarrow \infty$ one easily 
shows that $\phi(k,h) = o(\log k)$. On the other hand, if $k$ is 
fixed and $h \rightarrow \infty$ one gets a bound $\phi(k,h) \leq
(1+o(1))h$. It remains to investigate what the best-possible bounds could be,
for example in these two situations. Notice that the latter 
bound is not optimal for $k = 1$, by the main result of \cite{2}.  
\end{rek}
 
\setcounter{equation}{0}
\section{Bases without any essentialities}

The idea for our construction 
is quite simple. Let $A$ be a basis for $\mathbb{N}_0$ of some order
and suppose $P$ is an essentiality of $A$. Then for every $x \in
P$, the set $A_{x,P} := (A \backslash P) \cup \{x\}$ is once again a basis for
$\mathbb{N}_0$. Thus $x$ is an essential element of $A_{x,P}$. By Lemma 2.1(i),
this means that the set $A \backslash P$ is contained in a 
non-trivial arithmetic progression, i.e.: there exists an integer 
$d > 1$ and $c \in \{0,1,...,d-1\}$ such that $a \equiv c \; 
({\hbox{mod $d$}})$ for every $a \in A \backslash P$. Thus, a basis $A$
for $\mathbb{N}_0$ possesses no essentialities if it has the following 
property :
\\
\\
\noindent
{\em If $B \subseteq A$ is still a basis, of some order, 
then for every $d \geq 2$ and $c \in \{0,1,...,d-1\}$, $B$ contains 
infinitely many elements which are congruent to $c$ modulo $d$.}
\\
\\
\noindent
For want of a better term, a basis with this property shall be called
{\em devolved}. Thus it just remains to construct devolved 
bases, and this we shall now do.
\\
\\
\noindent
Let $h \geq 2$ be given. We construct a devolved
basis $A$ of order $h$. 
The idea is to have 
\be
A = \mathcal{I} \sqcup \mathcal{J}
\ee
where 
\be
\mathcal{I} = \bigsqcup_{n=1}^{\infty} \mathcal{I}_n, \;\;\;\; 
\mathcal{J} = \bigsqcup_{n=1}^{\infty} \mathcal{J}_n
\ee
and the following hold :
\\
\\
{\bf E1.} Each $\mathcal{I}_n$ is a finite interval, say $\mathcal{I}_n = 
[r_n,R_n]$.
\\
{\bf E2.} Each $\mathcal{J}_n$ is a finite arithmetic progression, say
$[s_n,S_n] \cap (c_n + d_n \mathbb{Z})$.
\\
{\bf E3.} $r_1 = 0$ and, for every $n \geq 1$, $r_n < R_n < s_n < S_n < r_{n+1}$.
\\
{\bf E4.} For every $d \geq 2$ and $c \in \{0,1,...,d-1\}$, there
are infinitely many $n \geq 1$ such that $J_n \subseteq c + d\mathbb{Z}$. 
\\
\\
We need to show that an appropriate choice of the parameters 
$r_n,R_n,s_n,S_n$ yields a set $A$ which is a devolved 
basis of order $h$. First of all, let $\mathbb{X}$ be 
the set of all ordered integer triples $(c,d,t)$, where 
$t \geq 1$, $d \geq 2$ and $c \in \{0,1,...,d-1\}$. This is a countable 
set, so let $\mathscr{O}$ be any well-ordering of 
it. We have quite a lot of freedom in the choices of the above 
parameters, but something specific that works is the following recursive 
recipe :
\\
\\
{\bf Step 1.} $R_1 := 2$, $\mathbb{X}^{*} := \{ \}$, $n := 2$, $q := h+n$.
\\
{\bf Step 2.} Let $(c_n,d_n,t_n)$ 
be the least element of $\mathbb{X} \backslash 
\mathbb{X}^{*}$, as defined by the ordering $\mathscr{O}$, such that 
$d_n \leq (h-1)(R_{n-1} - r_{n-1}) + 1$. Take $s_n$ to be the first number
greater than $R_{n-1}$ satisfying $s_n \equiv c_n \; ({\hbox{mod $d_n$}})$ and
take $J_n := [s_n,S_n] \cap (c_n + d_n \mathbb{Z})$, where $S_n$ is the 
smallest number greater than $qs_n$ such that $S_n \equiv c_n 
\; ({\hbox{mod $d_n$}})$. 
\\
{\bf Step 3.} Update $n := n+1$. Take $r_n := S_{n-1} + 1$ and $R_n := 
hr_n$. Update $\mathbb{X}^{*} := \mathbb{X}^{*} \cup 
\{(c_{n-1},d_{n-1},t_{n-1})\}$ 
and go to Step 2. 
\\
\\
It is straightforward to check that our choices ensure
that the set $A$ given by (3.1) and (3.2) satisfies the properties
{\bf E1} through {\bf E4}. It remains to verify the following two claims :
\\
\\
{\sc Claim 1} : $A$ is a basis of order $h$.
\\
{\sc Claim 2} : $A$ is devolved.
\\
\\
{\em Proof of Claim 1.} For each $n \geq 1$ let  
\be
A_n := \left( \bigsqcup_{k=1}^n I_k \right) \sqcup \left( 
\bigsqcup_{k=1}^{n-1} J_k \right).
\ee
Clearly, $hA_1 = [0,hR_1]$. Suppose for some $n \geq 1$ that 
$hA_n = [0,hR_n]$. The choice of $s_n$ guarantees that there is at least one
representation
\be
hR_n + 1 = s_n + \alpha_1 + \cdots + \alpha_{h-1},
\ee
where $\alpha_1,...,\alpha_{h-1} \in A_n$. Then the choice of $d_n$ ensures
that, at least for every $x \in (hR_n,S_n+(h-1)R_n]$, there is at least one
representation 
\be
x = \beta + \alpha_1 + \cdots + \alpha_{h-1},
\ee
where $\beta \in J_n$ and $\alpha_1,...,\alpha_{h-1} \in A_n$. Finally, then, 
the choices of $r_{n+1}$ and $R_{n+1}$ ensure that $hA_{n+1} = [0,hR_{n+1}]$.
This completes the proof of our first claim.
\\
\\
{\em Proof of Claim 2.} Let $n \geq 1$. Since $S_n > (h+n)s_n$, 
any representation of the number $S_n$ as a sum of at most 
$h+n$ elements of $A$ must contain an element from $J_n$. Now let 
$B$ be a subset of $A$ which is still a basis of some order. It follows 
immediately that
$B$ must intersect all but finitely many of the sets 
$J_n$. But then, by property {\bf E4}, $A$ must be devolved.

\ \\

\end{document}